\documentclass[11pt, a4paper, oneside, DIV11, final, pagebackref]{amsart}
\usepackage[notref,notcite]{showkeys}
\usepackage[latin1]{inputenc}
\usepackage{amssymb}
\usepackage[T1]{fontenc}
\usepackage[stretch=10]{microtype}
\usepackage{fixltx2e}
\usepackage{fix-cm}
\usepackage{mathtools}
\usepackage[DIV11, headinclude=true]{typearea}
\usepackage{hyperref}
\usepackage{comment}

\includecomment{arxivversion}
\excludecomment{submittedversion}

\providecommand{\href}[2]{#2}

\providecommand*{\backref}{}
\providecommand*{\backrefalt}{}
\renewcommand*{\backref}[1]{}
\renewcommand*{\backrefalt}[4]{%
	\ifcase #1 %
	\or
	  Cited page~#2.
	\else
	  Cited pages~#2.
	\fi
}

\newcommand{\boF}{\mathcal{F}}
\newcommand{\E}{\mathbf{E}}
\newcommand{\PP}{\mathbf{P}}
\newcommand{\I}{\mathbb{I}}
\newcommand{\N}{\mathbb{N}}
\newcommand{\Z}{\mathbb{Z}}
\newcommand{\R}{\mathbb{R}}
\newcommand{\F}{\mathbb{F}}

\newcommand{\dd}{\mathop{}\!\mathrm{d}}
\newcommand{\norm}[1]{\left\| #1 \right\|}
\newcommand{\st}{\::\:}
\newcommand{\given}{\:\mid\:}
\DeclareMathOperator{\argth}{artanh}
\DeclareMathOperator{\supp}{supp}
\DeclareMathOperator{\dist}{dist}
\DeclareMathOperator{\bnd}{bnd}
\newcommand{\abs}[1]{\left|#1\right|}
\newcommand{\lgth}{\abs}
\newcommand{\pireg}{\pi_{\mathrm{reg}}}

\newcommand{\coloneqq}{\mathrel{\mathop:}=}
\renewcommand{\leq}{\leqslant}
\renewcommand{\geq}{\geqslant}
\renewcommand{\epsilon}{\varepsilon}
\renewcommand{\phi}{\varphi}

\newtheorem{thm}{Theorem}[section]
\newtheorem{prop}[thm]{Proposition}

\newtheorem{lem}[thm]{Lemma}
\newtheorem{cor}[thm]{Corollary}
\newtheorem*{prop*}{Proposition}
\newtheorem*{cor*1}{Corollary~\ref{characterize}}
\newtheorem*{cor*2}{Corollary~\ref{one-ended}}

\theoremstyle{definition}
\newtheorem{example}[thm]{Example}
\newtheorem{rmk}[thm]{Remark}

\numberwithin{equation}{section}

\title{Sharp lower bounds for the asymptotic entropy of symmetric random walks}
\author{S\'ebastien Gou\"ezel, Fr\'ed\'eric Math\'eus, Fran\c{c}ois Maucourant}

\address{S\'ebastien Gou\"ezel, IRMAR, Universit\'e Rennes 1,
35042 Rennes Cedex, France} \email{sebastien.gouezel@univ-rennes1.fr}
\address{Fr\'ed\'eric Math\'eus, Universit\'e de Bretagne Sud, L.M.B.A., UMR 6205,
BP 573, 56017 Vannes, France} \email{Frederic.Matheus@univ-ubs.fr}
\address{Fran\c{c}ois Maucourant, IRMAR, Universit\'e Rennes 1,
35042 Rennes Cedex, France}
\email{francois.maucourant@univ-rennes1.fr}
\date{February 5, 2014}

\keywords{Random walk, countable group, entropy, spectral radius, drift, volume growth, Poisson boundary}
\subjclass[2010]{05C81, 60B15, 60J50}

\begin{document}

\begin{abstract}
The entropy, the spectral radius and the drift are important
numerical quantities associated to random walks on countable groups.
We prove sharp inequalities relating
those quantities for walks with a finite second moment, improving
upon previous results of Avez, Varopoulos, Carne, Ledrappier. We also
deduce inequalities between these quantities and the volume growth of
the group. Finally, we show that the equality case in our inequality
is rather rigid.
\end{abstract}

\maketitle

\section{Introduction and main results}

Let $\Gamma$ be a countable group and $\mu$ a probability measure on
$\Gamma$. The right random walk associated with the pair
$(\Gamma,\mu)$ is the Markov chain on $\Gamma$ whose transition
probabilities are defined by $p(x,y) = \mu(x^{-1}y)$. A realization
of the random walk starting from the identity is given by $X_0 = e$
and $X_n = \gamma_1 \dotsm \gamma_n$ where $(\gamma_i)_i$ is an
independent sequence of $\Gamma$-valued $\mu$-distributed random
variables. The law of $X_n$ is the $n$-fold convolution $\mu^{*n}$ of
$\mu$.

Let $\lgth{\cdot}=\dist(\cdot,e)$ denote the distance to the identity, for
a proper left-$\Gamma$-invariant distance $\dist(\cdot,\cdot)$ on
$\Gamma$ (in examples, we will choose implicitly the word length with
respect to a finite symmetric set of generators $S$). Several
numerical quantities were introduced to describe the asymptotic
behavior of $X_n$. The asymptotic entropy $h$, the spectral radius
$\rho$ and the drift (or rate of escape) $\ell$ of the random walk with
respect to $\lgth{\cdot}$ are defined by
\begin{align*}
  h     & =  \lim_n -\frac{1}{n}\sum_g \mu^{*n}(g)\log \mu^{*n}(g), \\
  \rho  & =  \limsup_n \root n \of {\mu^{*n}(e)}\leq 1, \\
\textrm{and}\quad
  \ell  & =  \lim_n \frac{1}{n}\sum_g \lgth{g} \mu^{*n}(g).
\end{align*}
The asymptotic entropy is well-defined if the entropy
$H(\mu) = -\sum_{g\in\Gamma} \mu(g) \log \mu(g)$ is finite.
No assumption on the measure $\mu$ is required to define the spectral
radius $\rho$. The drift $\ell$ is well-defined if $\mu$ has finite first
moment. Note that if the cardinality of the balls $B(e,n)$ grows at most
exponentially, then finiteness of the first moment implies finiteness
of the entropy, see~\cite{Der86,Kaim98}).

Assume that $\mu$ is supported on a finite set of generators $S$, and
that $\dist$ is the corresponding word distance. When the set $S$
has $2d$ elements, the drift is bounded by the drift of the simple
random walk in a regular tree with valence $2d$, i.e., $1-1/d$.
However, if there are a lot of relations in the group, the walk is
more likely to come back closer to the identity, and one would expect
a smaller drift. In this direction, it is more relevant to consider
the volume growth $v = \lim_n \frac{1}{n}\log \#B(e,n)$ of $\Gamma$
with respect to $S$ rather than merely the number of generators: one
may expect that a bound on $v$ implies a
bound on the drift, of the form $\ell \leq f(v)$ for some function
$f$ taking values in $[0,1)$. Such an inequality is surprisingly hard
to prove directly. Our first result answers this question, for an
explicit function $f$. A similar discussion holds for the spectral
radius (one can bound $\rho$ from below using the number of
generators, by $\sqrt{2d-1}/d$, see Kesten~\cite{Kes59}, but bounds
involving $v$ are harder to come with).

Our inequalities hold for measures with finite second moment;
throughout the paper, we will write $M_2 (\mu) \coloneqq \left(\sum_g
\lgth{g}^2 \mu(g)\right)^{1/2}$ for the $\ell^2$-norm with respect to
the measure $\mu$ of the distance to the identity.

\begin{thm}
\label{thm_v}
Let $\Gamma$ be a countable group with a proper left-invariant
distance, such that $v=\liminf_n \frac{1}{n}\log \#B(e,n)$ is finite.
Let $\mu$ be a symmetric probability measure with a finite second
moment on $\Gamma$. Denote by $\tilde\ell=\ell/M_2(\mu)$ and $\tilde
v = M_2(\mu)v$ the drift and the growth for the distance $\widetilde
\dist(g,h)=\dist(g,h)/M_2(\mu)$. The following inequalities hold:
  \begin{equation*}
  \tilde \ell \leq \tanh(\tilde v/2)
  , \quad
  h           \leq \tilde v \tanh(\tilde v/2)
  ,\quad
  \rho        \geq 1/\cosh(\tilde v/2).
  \end{equation*}
\end{thm}

These inequalities are consequences of other inequalities relating
$h$ to $\ell$ and $\rho$:

\begin{thm}
\label{main_thm}
Let $\mu$ be a symmetric probability measure with finite entropy on a
countable group $\Gamma$ with a proper left-invariant distance. Then
  \begin{equation}
  \label{eq_main_rho}
  2\sqrt{1-\rho^2} \argth\sqrt{1-\rho^2}\leq h.
  \end{equation}
Moreover, if $\mu$ has a finite second moment,
  \begin{equation}
  \label{eq:main_ell}
  2 \tilde\ell \argth \tilde\ell \leq h
  \end{equation}
where $\tilde\ell = \ell/M_2(\mu)$.
\end{thm}

The first theorem is a consequence of the second one:
\begin{proof}[Proof of Theorem~\ref{thm_v} using
Theorem~\ref{main_thm}] The entropy satisfies the so-called
``fundamental inequality'' $h\leq \ell v=\tilde \ell \tilde v$,
by~\cite{Guiv80}. Since $2\tilde \ell \argth \tilde\ell \leq h$
by~\eqref{eq:main_ell}, this yields $2\argth(\tilde \ell) \leq \tilde
v$, hence $\tilde \ell \leq \tanh(\tilde v/2)$. Since $h\leq
\tilde\ell \tilde v$, we deduce that $h\leq \tilde v \tanh(\tilde
v/2)$. Last, we remark that $r=1/\cosh(\tilde v/2)$ satisfies
  \begin{equation*}
  2 \sqrt{1-r^2} \argth \sqrt{1-r^2} = 2 \tanh(\tilde v/2) \argth(\tanh(\tilde v/2))
  = \tilde v \tanh(\tilde v/2).
  \end{equation*}
We have already proved that this is larger than or equal to $h$.
Together with~\eqref{eq_main_rho} and the fact that $t\mapsto
2\sqrt{1-t^2} \argth\sqrt{1-t^2}$ is non-increasing, this gives
$\rho\geq r$, as claimed.
\end{proof}

The inequalities of Theorem~\ref{main_thm} have several predecessors.
The first lower bound for the asymptotic entropy is due to
A.~Avez~\cite{Avez76}, who proved that $h \geq -2\log\rho$. More
recently, Ledrappier~\cite{ledrappier_sharp_entropy} showed that
$h\geq 4(1-\rho)$. Those two inequalities are not comparable,
Ledrappier's being stronger for $\rho$ close to $1$ but weaker for
$\rho$ close to $0$. The inequality~\eqref{eq_main_rho} is a common
strengthening of both inequalities of Avez and Ledrappier, since the
left-hand side of~\eqref{eq_main_rho} is larger than $\max(-2\log
\rho, 4(1-\rho))$ (and asymptotic to $-2\log\rho$ when $\rho$ tends
to $0$, and to $4(1-\rho)$ when $\rho$ tends to $1$). This statement
may not be obvious from the formula~\eqref{eq_main_rho}, but it
follows readily from the analysis of this function that we will have
to do later on
\begin{arxivversion}
(see in particular Lemma~\ref{lem_etude_FGmoinsun}, or
Corollaries~\ref{cor_chebyshev} and~\ref{cor_ledrappier}).
\end{arxivversion}
\begin{submittedversion}
(see in particular Lemma~\ref{lem_etude_FGmoinsun}).
\end{submittedversion}

\medskip

Lower bounds for the asymptotic entropy $h$ involving the drift
$\ell$ were also considered. Varopoulos~\cite{Varop85} and
Carne~\cite{Carne85} proved that
\[
\forall g\in \Gamma, \mu^{*n}(g)\leq 2\exp\biggl[-\frac{\lgth{g}^2}{2nk^2}\biggr]
\]
where $k$ is the radius of the smallest ball containing the support
of $\mu$. The consequence for $h$ and $\ell$ becomes $h\geq \ell^2 /
2k^2$.
\begin{arxivversion}
Actually, Carne's estimate can be improved. A careful inspection of
his proof enabled J.~L{\oe}uillot to prove the following: with the same
notations
\[
\forall g\in \Gamma, \mu^{*n}(g)\leq 2\rho^n \exp\biggl[-\frac{n}{2} A\biggl(\frac{\lgth{g}}{nk}\biggr)\biggr]
\]
where $A$ is defined, for $x\in [0,1)$, by $A(x) =
(1+x)\log(1+x)+(1-x)\log(1-x)$. Using Jensen inequality and the
convexity of $A$, he deduced in~\cite{Loe11} that $h\geq A(\ell/k)/2
- \log\rho$. Using Theorem~\ref{main_thm}, one can improve this
inequality by a factor of $2$ and replace $k$ by $M_2(\mu)\leq k$,
see Corollary~\ref{cor_chebyshev}.
\end{arxivversion}
\begin{submittedversion}
Theorem~\ref{main_thm} improves this inequality in several ways: it
replaces $k$ by $M_2(\mu)$ (allowing measures with infinite support),
it replaces the square function with the better function $2 x
\argth(x)$ (which is $x^2 + o(x^2)$ at $0$, and strictly larger than
$x^2$ away from $0$), and it gains a multiplicative factor of $2$.
\end{submittedversion}

Recently, A.~Erschler and A.~Karlsson proved in~\cite{EK12} that
$h\geq \ell^2 / C(\mu)$ still holds for symmetric probability
measures with finite second moment giving nonzero probability to the
identity, where $C(\mu)$ depends on $\mu$ (the main dependency is on
$\mu(e)>0$ and on $M_2(\mu)<\infty$).

\medskip

Theorem~\ref{main_thm} owes a lot to~\cite{ledrappier_sharp_entropy}
and~\cite{EK12}: our investigations started when we tried to
understand and sharpen the arguments in those two papers. The proofs
in these articles are given inside the group, studying the random
walk at finite time (or a poissonized version of the random walk
in~\cite{ledrappier_sharp_entropy}). It turns out that
Theorem~\ref{main_thm} can be proved following the same strategy.
However, an (essentially equivalent) proof can also be given using
various boundaries (the Poisson boundary for the inequality involving
$\rho$, the horocycle boundary for the inequality involving $\ell$).
This proof has the advantage of avoiding limits completely, making it
possible to characterize the equality case in our inequalities (see
Proposition~\ref{equal} below). Therefore, we will concentrate mainly
on the proof using boundaries: at the beginning of
Section~\ref{sec:proofs}, we will quickly sketch the proofs inside
the group, without giving all the details, and the rest of
Section~\ref{sec:proofs} will be devoted to a complete proof using
boundaries.
\begin{submittedversion}
Before this, Section~\ref{sec:examples} is devoted to more examples
and comments.
\end{submittedversion}
\begin{arxivversion}

The other sections of the paper are organized as follows:
Section~\ref{sec:examples} is devoted to more examples and comments,
and Section~\ref{sec:chebyshev} contains a discussion of a corollary
of Theorem~\ref{main_thm}, together with an additional elementary
proof more in the spirit of Carne-Varopoulos that we find interesting
in its own right. This section is removed in the published version.
\end{arxivversion}

\medskip

Let us stress that inequalities similar to the results of
Theorem~\ref{main_thm} have been known for a longer time for
Brownian motion on cocompact Riemannian manifolds (see for instance~\cite{Kaim86,
ledrappier_continuous, ledr10}): infinitesimal inequalities are
available and make for a simpler result.

\section{Examples and comments}
\label{sec:examples}

Let us first note that the conclusion of Theorem~\ref{main_thm} does
not hold any more if the measure $\mu$ is not symmetric. For
instance, for the random walk on $\Z$ given by $\mu = p\delta_{-1} +
(1-p)\delta_{+1}$, one has $h=0$ while $\ell = \abs{2p-1}$ and $\rho
= 2\sqrt{p(1-p)}$. When $p\not=1/2$, one gets $\ell>0$ and $\rho<1$,
hence Theorem~\ref{main_thm} does not hold in this case.

\begin{example}
Let $\Gamma = \F(a_1,\dotsc,a_d)$ be the free non-abelian group over
$\{a_1,\dotsc,a_d\}$, with its usual word distance. We consider the
simple random walk on $\Gamma$, i.e., we take for $\mu$ the uniform
measure on $S = \{a_1,\dotsc,a_d\}^{\pm 1}$. In this case, one can
easily compute all the quantities involved in Theorems~\ref{thm_v}
and~\ref{main_thm}. Indeed, one has
\begin{itemize}
\item $\ell = 1-1/d$, since at each step away from the identity
    there is probability $1-1/(2d)$ to go further to infinity,
    and $1/(2d)$ to come back.
\item $\rho = \frac{\sqrt{2d-1}}{d}$ since the number of words
    back to the identity at time $2n$ has a generating series
    $\frac{d\sqrt{1-4z^2(2d-1)}-d+1}{1-4d^2z^2}$, with first
    singularity at $z=1/(2\sqrt{2d-1})$, see~\cite[Lemma~1.24]{Woe00}.
\item $h = (1-1/d)\log(2d-1)$. This follows for instance from the
    description of the Poisson boundary as the set of infinite
    reduced words $b_0 b_1\dotsm$, with the measure $\nu$ giving
    mass $1/(2d (2d-1)^{n-1})$ to any cylinder of length $n$, and
    from the formula~\eqref{eq:entropy_Poisson} below giving the
    entropy as an integral over the Poisson boundary of the
    logarithm of the Radon-Nikodym derivative of the group
    action. See~\cite{ledrappier_sharp_entropy} or the proof of
    Corollary~\ref{characterize} for more details.
\item $v=\log(2d-1)$. Indeed, the sphere of radius $n$ has
    cardinality $2d (2d-1)^{n-1}$.
\end{itemize}
It follows that, in this case, all the inequalities in
Theorems~\ref{thm_v} and~\ref{main_thm} are equalities. This shows in
particular that the inequalities of those theorems are sharp for
infinitely many values of the entropy.
\end{example}

\begin{example}
\label{example:h<lv}
From the free group, one can construct other examples where equality
holds in Theorem~\ref{main_thm}. For instance, let $H$ be a finite
group and let $\F$ be a free group on finitely many generators
$\{a_1,\dotsc, a_d\}$. In $\Gamma = H\times\F$, consider the
generating set $S = \{(x,a_i^{\pm 1}) \st x\in H, i\in \{1,\dotsc,
d\} \}$. The simple random walk on $(\Gamma,S)$ projects to the
simple random walk on the free factor $\F$, and those random walks
have the same drift, entropy and spectral radius. Since equality
holds in Theorem~\ref{main_thm} for $\F$, it follows that is also
holds for $(\Gamma, S)$.

More generally, consider an exact sequence
  \begin{equation}
  \label{eq:exact_sequence}
  1 \to H \to \Gamma \to \F \to 1
  \end{equation}
where $\F$ is a group whose Cayley graph with respect to some
generating system is a tree, and a probability measure on the set of
generators of $\Gamma$ that projects to the uniform measure on the
generators of $\F$. If the drift, entropy and spectral radius of the
random walk on $\Gamma$ are the same as on $\F$ (this is for instance
the case if $H$ has subexponential growth), then equality holds in
Theorem~\ref{main_thm} for the random walk on $\Gamma$. Concretely,
one may consider for instance any semi-direct product $\Gamma = \Z^k
\rtimes \F$ where $\F$ is a free subgroup of $GL(k,\Z)$. For another
example, let $\Gamma'=\Z \wr \Z/3\Z$ with its standard set of
generators $S'$, let $\Gamma=\Gamma' \times \F$ and let $S=\{(x,
a_i^{\pm 1}) \st x\in S', i\in \{1,\dotsc, d\} \}$ (this set
generates $\Gamma$ since we use $\Z/3\Z$ -- with $\Z/2\Z$ instead, it
would generate an index two subgroup of $\Gamma$). Since the simple
random walk on $\Gamma'$ has zero entropy and drift
(see~\cite{kaim_vershik}), equality in Theorem~\ref{main_thm} holds
for the simple random walk on $\Gamma$. This example is interesting
since the volume growth $v$ of $\Gamma$ is strictly larger than the
volume growth in the free group as $\Gamma'$ has exponential growth.
Hence, $h<\ell v$, showing that equality in Theorem~\ref{main_thm}
does not imply equality in the fundamental inequality. There is no
implication in the other direction either, see the discussion after
Corollary~\ref{characterize}.

We conjecture (but are unable to prove) that the above
situation~\eqref{eq:exact_sequence} is the only case where equality
holds in Theorem~\ref{main_thm}. Partial results in this direction
are given in Corollaries~\ref{characterize} and~\ref{one-ended}.
\end{example}

\begin{example}
Assume that $\Gamma$ is the fundamental group of a closed compact
surface of genus $2$. Consider the following presentation of
$\Gamma$:
\[
\Gamma = \langle a_1,a_2,b_1,b_2 \st [a_1,b_1][a_2,b_2] = 1 \rangle.
\]
The growth $v$ of $\Gamma$ with respect to the generating set
$S=\{a_1,a_2,b_1,b_2\}^{\pm 1}$ is explicitly known. Following Cannon
(see~\cite[\S~VI.A.8]{laHarpe}), it is the logarithm of an algebraic number, and
its value is $v = 1.9430254\dots$. Let $\mu$ be any symmetric
probability measure on $S$. Theorem~\ref{thm_v} gives
  \begin{equation*}
  \ell \leq 0.749368278, \quad
  h \leq 1.456041598,\quad
  \rho \geq 0.66215344.
  \end{equation*}
This is better than the naive estimates obtained using only the
number of generators, by comparing to the free group, giving
$\ell\leq 0.75$ and $h\leq 1.45944$ and $\rho\geq 0.66143$. Note that
the gain is not very important, but this is not surprising since
$\Gamma$ is very close to being free (the growth in the corresponding
free group is $\log(7)=1,945910\dots$, close to $v$ up to
$3.10^{-3}$).

Assume now that $\mu$ is the uniform measure on $S$. The best known
estimates on $\rho$ are $0.662420 \leq \rho \leq 0.662816$
(see~\cite{bartholdi_rho} and~\cite{nagnibeda_rho}). It follows that
our bound for $\rho$, although worse than Bartholdi's, is precise up to
$7.10^{-3}$, while the bound using the number of generators is
precise up to $14.10^{-3}$, i.e., twice worse. Using Nagnibeda's
upper bound for $\rho$, the inequality~\eqref{eq_main_rho} estimating
$h$ in terms of $\rho$ gives $h\geq 1.452903618$. Since $\ell \geq
h/v$, we also have $\ell\geq 0.747753281$. This proves that the upper
bounds we get for $\ell$ and $h$ are precise up to $2.10^{-3}$ and
$4.10^{-3}$, to be compared with the bounds using only the number of
generators that are precise up to $2.10^{-3}$ and $7.10^{-3}$: the
gain is very small for $\ell$, more significant for $h$.
\end{example}

Let $\Gamma$ be a countable group, and let $\mu$ be a symmetric
probability measure on $\Gamma$ whose support generates $\Gamma$,
with finite entropy. A $(\Gamma,\mu)$-space is a probability space
$(\mathcal{B},\nu)$ endowed with a $\Gamma$-action, such that the
probability $\nu$ is $\mu$-stationary, i.e.,
\[
\nu = \mu * \nu \stackrel{\text{def}}{=} \sum_{\gamma\in\Gamma}\mu(\gamma)\gamma_*\nu.
\]
In particular, $\gamma_*\nu$ is absolutely continuous with respect to
$\nu$, for every $\gamma$ in the subgroup generated by the support of
$\mu$, which we assume to coincide with $\Gamma$.

A particularly interesting $(\Gamma,\mu)$-space is its \emph{Poisson
boundary}, that we will denote by $(\mathcal{B}_0, \nu_0)$: it is the
unique $(\Gamma,\mu)$-space parameterizing harmonic functions.
Equivalently, it can be seen as the exit boundary of the random walk
on the group, made of the events that only depend on the tails of
infinite trajectories of the random walk (see~\cite{kaim_vershik} for
the equivalence and for several other definitions).

The Poisson boundary will play an important role in the proof of
Theorem~\ref{main_thm}. It will follow from the proof that the
equality case in this theorem implies a rigid behavior of the Poisson
boundary:

\begin{prop}\label{equal}
On a countable group $\Gamma$ with a proper left-invariant distance,
consider a symmetric probability measure $\mu$ with finite second
moment. Assume that one of the inequalities of Theorem~\ref{main_thm}
is an equality. Then, on the Poisson boundary $(\mathcal{B}_0,
\nu_0)$ of $(\Gamma,\mu)$, the Radon-Nikodym derivative
$\frac{\dd\gamma^{-1}_*\nu_0}{\dd \nu_0}(\xi)$ takes only two values
$e^\alpha$ and $e^{-\alpha}$, $\mu\otimes \nu_0$ almost surely.
\end{prop}

There can be no converse to this proposition for the
inequality~\eqref{eq:main_ell} involving the distance, since the
conclusion of the proposition does not involve the distance. For
instance, consider in the free group on two generators $a$ and $b$ a
family of distances $d_\epsilon$ giving weight $1$ to $a$ and
$\epsilon$ to $b$. If $\mu$ is the uniform measure on the generators,
then equality holds in Theorem~\ref{main_thm} for $d_1$, so that the
conclusion of Proposition~\ref{equal} holds. On the other hand, the
inequality~\eqref{eq:main_ell} is strict for $\tilde \ell$ defined
using $d_\epsilon$, if $\epsilon\not=1$ (one gets $\tilde
\ell_\epsilon = (1+\epsilon)/
\sqrt{8(1+\epsilon^2)}<1/2=\tilde\ell_1$). We do not know if there is
a converse to Proposition~\ref{equal} regarding the
inequality~\eqref{eq_main_rho} about the spectral radius.

\medskip

Proposition~\eqref{equal} (the proof of which is given at the end of
Section~\ref{sec:proofs}) makes it possible to describe precisely
some situations where equality can or cannot occur. We should stress
that this is very different from the fundamental inequality $h\leq
\ell v$, where equality is much more difficult to characterize
(see~\cite{ledr01} for the free group, \cite{MaMa} when $\Gamma$ is a
free product of finite groups, or~\cite{BHM} for several
characterizations of the equality in terms of quasi-conformal
measures on the boundary when $\Gamma$ is a word-hyperbolic group).

Note that, in the previous proposition (and in the corollaries
below), the choice of the distance is not important: if there is
equality in~\eqref{eq:main_ell} for \emph{any} proper left-invariant
distance, then the conclusions of Proposition~\ref{equal} hold.

\begin{cor}
\label{characterize}
Assume that $\Gamma = \Gamma_1 * \dotsm * \Gamma_q$ ($q\geq 2$) is a
free product of finitely generated groups $\Gamma_i$ with finite
generating sets $S_i$. Let $\mu$ be a symmetric probability measure
with support equal to $S=\bigsqcup S_i$. Assume that one of the
inequalities of Theorem~\ref{main_thm} is an equality. Then the
Cayley graph of each $\Gamma_i$ with respect to $S_i$ is a regular
tree (i.e., $\Gamma_i$ is a free product of finitely many factors
$\Z$ and $\Z/2\Z$), the Cayley graph of $\Gamma$ with respect to $S$
is also a regular tree, and $\mu$ is the uniform measure on $S$.
\end{cor}

For instance, consider the modular group $\Gamma = \Z/2\Z * \Z/3\Z =
\{1,a\}*\{1,b,b^2\}$ and a symmetric probability measure $\mu_p =
p\delta_a + \frac{1-p}{2}(\delta_b + \delta_{b^2})$. Then, for all
$p\in(0,1)$, one has $h = \ell v$ ~\cite{MaMa} but the inequalities
are strict in Theorem~\ref{main_thm}. Together with the example of
$(\Z \wr \Z/3\Z) \times \F$ (see Example~\ref{example:h<lv} above),
this shows that equality in Theorem~\ref{main_thm} and in the
fundamental inequality $h\leq \ell v$ are independent.

As far as the free group $\F_d=\Z*\dotsm * \Z$ is concerned, the
above corollary says that the simple random walk is the only
symmetric nearest neighbor random walk for which equality holds in
Theorem~\ref{main_thm}.

\begin{proof}[Proof of Corollary~\ref{characterize}]
For $u\in \Sigma=\bigsqcup \Gamma_i\setminus \{e\}$, write
$\overline{u} = i$ if $u\in \Gamma_i$. A -- finite or infinite --
word $u_1 u_2 \dotsm $ over the alphabet $\Sigma$ is \emph{reduced}
if $\overline{u_i} \neq \overline{u_{i+1}}$. The group $\Gamma$ is
the set of finite reduced words over $\Sigma$ (the identity is the
empty word) endowed with the composition law which is the
concatenation with possible simplification at the contact point.

Denote by ${\mathcal E}(\Gamma)$ the \emph{space of ends} of
$\Gamma$. Let $\mu$ be a symmetric probability measure on $\Gamma$
with support equal to $S$. Then there exists a unique probability
measure $\nu_0$ on ${\mathcal E}(\Gamma)$ which is $\mu$-stationary,
and the space $({\mathcal E}(\Gamma),\nu_0)$ is (a realization of)
the Poisson boundary of $(\Gamma,\mu)$ (see~\cite{Woe89,Woe93} and
also~\cite{Kaim00}). The set $\partial \Gamma$ of right infinite
reduced words $\xi = \xi_1\xi_2 \dotsm$ over $\Sigma$ is a
$\Gamma$-invariant subset of ${\mathcal E}(\Gamma)$ with full
$\nu_0$-measure.

For $a\in \Gamma$, denote by $q(a) = \PP(\exists n, X_n=a)$ the
probability that the random walk ever reaches $a$. For $a\in \Sigma$
and $\xi = \xi_1\xi_2 \dotsm \in \partial \Gamma$, the Radon-Nikodym
derivative $c_0(a,\xi)=\frac{\dd a^{-1}_*\nu_0}{\dd \nu_0}(\xi)$
satisfies
\begin{equation}\label{cocycle_RN}
c_0(a,\xi)=\begin{cases} q(a) & \text{if } \overline{\xi_1} \neq \overline{a} \\
              q(a\xi_1)/q(\xi_1)  & \text{if } \overline{\xi_1} = \overline{a}
\end{cases}\quad,
\end{equation}
see~\cite{DyMa61},~\cite{ledr01} and~\cite{MaMa}.

\medskip

Assume that one of the inequalities of Theorem~\ref{main_thm} is an
equality. Proposition~\ref{equal} provides a real number $\alpha \geq
0$ such that $c_0(a,\xi) \in \{e^{\alpha},e^{-\alpha}\}$ for
$\mu\otimes\nu_0$-almost every $(a,\xi) \in \Sigma \times \partial
\Gamma$. Since $q(a)<1$ as the random walk on a free product is
transient, Equation~\eqref{cocycle_RN} implies that $\alpha > 0$ and
$q(a) = e^{-\alpha}$ for all $a\in S$.

Consider two elements $a,b\in S_i$ (possibly with $a=b$), with
$ab\not= e$. The second case in~\eqref{cocycle_RN} shows that
$q(ab)/q(b) \in \{e^\alpha, e^{-\alpha}\}$. Since $q(b) =
e^{-\alpha}$, this gives $q(ab) \in \{1, e^{-2\alpha}\}$. Since
$ab\not = e$ and the random walk is transient, we have $q(ab) < 1$,
hence $q(ab) = e^{-2\alpha}$. This gives
  \begin{align*}
  \PP(\exists m<n, X_m=a, X_n=ab) & = \PP(\exists m, X_m=a) \PP(\exists n, X_n = b)
  =q(a)q(b)
  \\ & = e^{-2\alpha} = q(ab) = \PP(\exists n, X_n = ab),
  \end{align*}
where we used the Markov property for the first equality. This shows
that almost every path from $e$ to $ab$ has to pass first through
$a$. Equivalently, whenever we write $ab$ as a product $s_1 \dotsc
s_n$ of elements of $S_i$, then some prefix $s_1 \dotsc s_m$ is equal
to $a$.

This implies that there is no nontrivial loop in the Cayley graph of
$\Gamma_i$ with respect to $S_i$: if there were such an injective
loop $e, a_1, a_1a_2, \dotsc, a_1a_2\dotsm a_{k-1}, a_1a_2\dotsm
a_{k-1}a_k=e$ (where all points but the first and last one are
distinct), then $a_1a_2=(a_3\dotsm a_{k})^{-1} = a_{k}^{-1}\dotsm
a_3^{-1}$. Since $S_i$ is symmetric, we have written $a_1 a_2$ as a
product of elements of $S_i$ that never reaches $a_1$ (since the loop
is injective), a contradiction. This shows that the Cayley graph of
$\Gamma_i$ with respect to $S_i$ is a regular tree, and therefore
that $\Gamma_i$ is a free product of finitely many factors $\Z$ and
$\Z/2\Z$.

The Cayley graph of $\Gamma$ with respect to $S$ is also a regular
tree. Since the probabilities of ever reaching any neighbor of the
origin are the same, so are the transition probabilities, hence $\mu$
is uniform on $\Sigma$.
\end{proof}

Corollary~\ref{characterize} characterizes the equality case for a
class of random walks on free groups, or more generally on some
virtually free groups. For hyperbolic groups, this is the only
situation where equality in our inequalities is possible:

\begin{cor}
\label{one-ended}
Let $\Gamma$ be a hyperbolic group which is not virtually free, and
let $\mu$ a finitely supported symmetric probability measure on
$\Gamma$ whose support generates $\Gamma$ as a semigroup. Then the
inequalities of Theorem~\ref{main_thm} are strict.
\end{cor}
\begin{proof}
Let $\Gamma$ be a hyperbolic group. If $\mu$ is a finitely supported
probability measure on $\Gamma$, then it follows from~\cite{Ancona}
that the Poisson boundary and the Martin boundary of $(\Gamma,\mu)$
can be identified with $(\partial \Gamma,\nu)$ where $\partial
\Gamma$ is the geometric boundary of $\Gamma$ and $\nu$ is the unique
$\mu$-stationary measure on $\partial \Gamma$ (it has full support
and no atom). In particular, the Martin kernel $c(g,\xi)=\frac{\dd
g^{-1}_*\nu}{\dd \nu}(\xi)$ is well defined and continuous on
$\partial\Gamma$.

Assume that one of the inequalities of Theorem~\ref{main_thm} is an
equality. From Proposition~\ref{equal}, for any $g$ in the support of
$\mu$, the continuous function $\xi \mapsto c(g,\xi)$ can only take
two values for $\xi$ in the support of $\nu$, which is the whole
space $\partial\Gamma$. Writing any element of the group as a finite
product of elements in the support of $\mu$, it follows that $\xi
\mapsto c(g,\xi)$ only takes finitely many values, for any $g\in
\Gamma$.

Suppose now that $\Gamma$ is not virtually free. It follows that the
boundary of $\Gamma$ is not totally disconnected, and moreover the
stabilizer of any nontrivial component $L$ of the boundary is a
subgroup $\Lambda$ of $\Gamma$, which is quasi-convex and therefore
hyperbolic, with limit set equal to $L$ (see the discussion on top of
Page~55 in~\cite{bowditch_2} for all these facts). Since $L$ is
nontrivial, $\Lambda$ is non-elementary. In particular, it contains
an element $g$ of infinite order, which is hyperbolic. The attractive
and repulsive points $g^+$ and $g^-$ of $g$ both belong to $L$.

The function $\xi \mapsto c(g,\xi)$ is continuous and takes finitely
many values. It follows that it is constant on $L$, equal to some
$c>0$. It is even equal to $c$ on a small neighborhood $U$ of $L$.

Let $V\subset U$ be a small neighborhood of $g^+$. Since $\nu$ has
full support, $\nu(V)>0$. As $\frac{\dd g^{-1}_*\nu}{\dd \nu}(\xi)=c$
on $V$, we get $c=\nu(gV)/\nu(V)$. Iterating $n$ times this equation,
we obtain $c^n=\nu(g^n V)/\nu(V)$. As $g^n V$ is attracted to $g^+$
and $\nu$ has no atom, we deduce that $c^n < 1$ for large enough $n$,
hence $c<1$. Arguing in the same way using $g^{-1}$ around $g^-$, we
get $c>1$. This is a contradiction.
\end{proof}

\section{Boundaries, and proofs of the main inequalities}

\label{sec:proofs}

In this section, we prove the two main inequalities of
Theorem~\ref{main_thm}. The proof can be equivalently given inside
the group (following the ideas of Ledrappier
in~\cite{ledrappier_sharp_entropy}), or using boundaries. We will
mainly use the latter point of view, since it allows for more
transparent and intrinsic arguments. Moreover, it gives more insights
about the equality case in our inequalities. Nevertheless, in the
first subsection, we will quickly sketch the proof inside the group,
for the sake of completeness and since it can motivate some
definitions on the boundary.

In this section, $\Gamma$ will always be a countable group with a
proper left-invariant distance, and $\mu$ a symmetric probability
measure on $\Gamma$ whose support generates $\Gamma$, with finite
entropy.

\subsection{Proofs inside the group}
\label{subsec:proofs_inside}

In this paragraph, we sketch proofs of the inequalities of
Theorem~\ref{main_thm} by arguing inside the group, following
Ledrappier~\cite{ledrappier_sharp_entropy}. We start with the
estimate involving $\ell$.

Let $L(n) = \sum \lgth{g} \mu^{* n}(g)$ be the average length at time
$n$ and $H(n) = -\sum \mu^{*n}(g) \log \mu^{* n}(g)$ the entropy at
time $n$. Their averages converge respectively to $\ell$ and $h$. If
one could compare (a function of) $L(n+1)-L(n)$ with $H(n+1)-H(n)$,
an inequality involving $\ell$ and $h$ would follow. It is possible
to estimate conveniently those quantities if $\mu(e)>0$ (this is one
of the assumptions in~\cite{EK12}) -- otherwise, one can replace
$\mu$ with $(\mu+\delta_e)/2$. However, this leads to suboptimal
inequalities.

A more efficient procedure, used by
Ledrappier~\cite{ledrappier_sharp_entropy}, is to consider a
poissonized version of the random walk, in continuous time, where
jumps along the trajectories of the initial random walk occur
according to a Poisson distribution. This ensures that, from time $t$
to $t+\epsilon$, there is a positive probability to stay at the same
place, even when $\mu(e)=0$. Formally, define probability measures
$\mu_t = e^{-t} \sum_{n=0}^\infty \frac{t^n}{n!} \mu^{* n}$, they
have the same entropy and drift as the sequence $\mu^{* n}$, i.e.,
$H(\mu_t)/t\to h$ and $L(\mu_t)/t\to \ell$. If $P$ denotes the Markov
operator associated to $\mu$, one has $\mu_t = e^{t(P-I)} \delta_e$.
Differentiating with respect to $t$, one gets $\mu_t'(x)=
((P-I)\mu_t)(x) = \sum_g (\mu_t(gx)-\mu_t(x)) \mu(g)$. This gives a
formula for the derivative of the entropy:
  \begin{equation*}
  H(\mu_t)'=-\sum_{x,g} \mu(g)(\mu_t(gx)-\mu_t(x)) (\log \mu_t(x) + 1).
  \end{equation*}
One would like to use this quantity to dominate functions of the
derivative of the drift, but this expression is not convenient to do
so since some terms in the sum can be negative, and one should take
care of subtle cancellations. Lemma 3
in~\cite{ledrappier_sharp_entropy} uses the symmetry of the measure
$\mu$ to rewrite the above formula, using a symmetrization procedure,
as
  \begin{equation}
  \label{eq:derive_Hmut}
  H(\mu_t)' = \frac{1}{2} \sum_{x,g} \mu(g)(\mu_t(gx)-\mu_t(x))(\log \mu_t(gx)-\log \mu_t(x)),
  \end{equation}
where the terms are all nonnegative.

The derivative of the drift $L(\mu_t)$ is given by
  \begin{equation*}
  L(\mu_t)' = \sum \lgth{x} \mu_t'(x) = \sum_{x,g} \lgth{x} \mu(g) (\mu_t(gx) - \mu_t(x)).
  \end{equation*}
It is clear that the derivative of the drift should be bounded by the
first moment of the measure, but this is not apparent from this
formula. However, using the symmetrization lemma of Ledrappier, one
gets
  \begin{equation}
  \label{eq:Lmut_prime}
  L(\mu_t)' = \frac{1}{2} \sum_{x,g} (\lgth{x}-\lgth{gx}) \mu(g) (\mu_t(gx) - \mu_t(x)),
  \end{equation}
where boundedness becomes more apparent. This formula is more suited
to computations. Indeed, let us estimate $\lgth{x}-\lgth{gx}$ by
$\lgth{g}$ and let us use Cauchy-Schwarz inequality with respect to
the measure $\mu(g)\mu_t(x)$ on $\Gamma\times \Gamma$, this yields a
bound
  \begin{equation*}
  \abs{L(\mu_t)'} \leq
  \frac{M_2(\mu)}{2} \left(\sum_{x,g} \mu(g) \mu_t(x) \left( \frac{\mu_t(gx)}{\mu_t(x)}-1\right)^2\right)^{1/2}.
  \end{equation*}
The latter sum has a flavor that is similar
to~\eqref{eq:derive_Hmut}, that can also be written as
  \begin{equation*}
  H(\mu_t)' =
  \frac{1}{2} \sum_{x,g} \mu(g) \mu_t(x) \left(\frac{\mu_t(gx)}{\mu_t(x)}-1\right)
  \log \left(\frac{\mu_t(gx)}{\mu_t(x)}\right).
  \end{equation*}
However, it is not possible to compare directly those two quantities
using Jensen's inequality: the problem is that the value of $(
\mu_t(gx)/\mu_t(x)-1)^2$ does not determine the value of $\abs{\log
(\mu_t(gx)/\mu_t(x))}$, since the symmetries of those quantities are
not the same (additive symmetry around $1$ for the former,
multiplicative symmetry around $1$ for the latter).

The solution to this problem is to estimate~\eqref{eq:Lmut_prime}
using Cauchy-Schwarz inequality with respect to a different
probability measure on $\Gamma\times \Gamma$, that is more symmetric
in a sense, namely $\mu(g) \cdot \frac{\mu_t(gx)+\mu_t(x)}{2}$. The
resulting bound for $L(\mu_t)'$ is
  \begin{equation*}
  \abs{L(\mu_t)'} \leq
  M_2(\mu) \left(\sum_{x,g} \mu(g) \frac{\mu_t(gx)+\mu_t(x)}{2}
    \left(\frac{\mu_t(gx)-\mu_t(x)}{\mu_t(gx)+\mu_t(x)}\right)^2\right)^{1/2}.
  \end{equation*}
The last factor in this expression can be written as
$(c-1)^2/(c+1)^2$ for $c=\mu_t(gx)/\mu_t(x)$; it is invariant under
the symmetry $c\mapsto c^{-1}$, just like $\abs{\log c}$. It follows
that this bound for $L(\mu_t)'$ can be compared to $H(\mu_t)'$,
applying Jensen's inequality to a suitable convex function, with
respect again to the probability measure $\mu(g) \cdot
\frac{\mu_t(gx)+\mu_t(x)}{2}$ on $\Gamma \times \Gamma$. The
inequality~\eqref{eq:main_ell} follows. The full details will be
given later on, in the proof using boundaries.

To prove the estimate involving $\rho$, one uses the function $f_t:
x\mapsto \mu_t(x)^{1/2}$ (which has unit norm in $\ell^2(\Gamma)$).
We have $\langle P f_t, f_t\rangle \leq \rho$ since $\rho$ is the
spectral radius of $P$ acting on $\ell^2(\Gamma)$. Hence,
  \begin{equation*}
  \sum_{x,g} \mu(g) \mu_t(gx)^{1/2} \mu_t(x)^{1/2} \leq \rho.
  \end{equation*}
This expression can not be directly compared
to~\eqref{eq:derive_Hmut}. One should instead use the (equivalent)
inequality $\langle (I-P)f_t, f_t\rangle \geq 1-\rho$: here, the
scalar product can be again written using the symmetrization lemma,
yielding
  \begin{equation}
  \label{eq:rho_inside}
  \frac{1}{2}\sum_{x,g} \mu(g) (\mu_t(gx)^{1/2} -\mu_t(x)^{1/2})^2 \geq 1-\rho.
  \end{equation}
Again, this expression has the same flavor as~\eqref{eq:derive_Hmut},
and can be compared to it using Jensen's inequality for a good convex
function and the probability measure $\mu(g) \cdot
\frac{\mu_t(gx)+\mu_t(x)}{2}$ that ensures the right symmetry of the
integrand. This is the only point of the argument where we depart
from Ledrappier, who instead relied on the elementary inequality
$(a-b) (\log a - \log b) \geq 4 (a^{1/2}-b^{1/2})^2$ (Lemma~2
in~\cite{ledrappier_sharp_entropy}), which readily gives $1-\rho \leq
\frac{1}{4} H(\mu_t)'$ thanks to~\eqref{eq:rho_inside}
and~\eqref{eq:derive_Hmut}. Again, details will be given later on
using boundaries.

In the next sections, we describe the same proofs, but using
boundaries. The poissonization procedure will not be needed, and
there will be no limit over $t$, all the computations will be direct.
This implies that the equality case in our inequalities can be
characterized, making it possible to prove Proposition~\ref{equal}.

\subsection{A symmetrization lemma}

It follows from the above proof inside the group that the two crucial
points are the symmetrization procedure (Lemma 3
in~\cite{ledrappier_sharp_entropy}) that makes it possible to always
manipulate nonnegative quantities, and the use of the symmetrized
measure $\mu(g) \cdot \frac{\mu_t(gx)+\mu_t(x)}{2}$ in the
inequalities of Cauchy-Schwarz and Jensen. In this subsection, we
describe the analogues of those tools in a general
$(\Gamma,\mu)$-space.

Let $(\mathcal{B},\nu)$ be a $(\Gamma,\mu)$-space, i.e., a
probability space endowed with a $\Gamma$-action for which $\nu$ is
stationary. The Radon-Nikodym cocycle
  \[
  c(\gamma,\xi)=\frac{\dd\gamma^{-1}_*\nu}{\dd \nu}(\xi)
  \]
allows us to define a measure on $\Gamma\times \mathcal{B}$:
  \[
  \dd m=\frac{c+1}{2}\dd \mu \dd \nu.
  \]
(It is the analogue of $\mu(g) \cdot \frac{\mu_t(gx)+\mu_t(x)}{2}$.)
One checks, by means of a change of variables, that $m$ is indeed a
probability measure; in fact, for every $\gamma$,
$\frac{c(\gamma,\xi)+1}{2} \dd \nu(\xi)$ is a probability measure on
$\mathcal{B}$. Moreover, since $\mu$ is symmetric, the measure $m$ is
invariant under the `flip' involution $(\gamma,\xi)\mapsto
(\gamma^{-1},\gamma\xi)$.

The following symmetrization lemma is the analogue of~\cite[Lemma
3]{ledrappier_sharp_entropy}. The term ``symmetrization'' comes from
the fact that the expression on the right hand side
of~\eqref{wxomicuvoiu} does not change under the flip involution. It
relies crucially on the symmetry of the measure $\mu$. Here and
throughout, we will write
  \begin{equation*}
  d(\gamma,\xi)=\frac{1-c(\gamma,\xi)}{1+c(\gamma,\xi)} \in (-1,1).
  \end{equation*}
Most quantities will be conveniently expressed in terms of $d$. In
particular, $c = (1-d)/(1+d)$. If $c$ is replaced by its inverse,
then $d$ is replaced by its opposite. Hence, quantities that are
invariant under the symmetry $c\mapsto c^{-1}$ give rise to even
functions when they are expressed in terms of $d$.

\begin{lem} \label{lemsym}
Consider an additive cocycle $f:\Gamma\times \mathcal{B} \to \R$,
i.e., a function satisfying $f(\gamma \gamma',\xi) = f(\gamma,
\gamma'\xi) + f(\gamma',\xi)$. If $f$ is integrable with respect to
$\dd\mu \dd\nu$, then
  \begin{equation}
  \label{wxomicuvoiu}
  \int_{\Gamma \times \mathcal{B}} f(\gamma,\xi)\dd \mu(\gamma)\dd \nu(\xi)
  =\int_{\Gamma \times \mathcal{B}} f(\gamma,\xi)d(\gamma,\xi) \dd m(\gamma,\xi).
  \end{equation}
\end{lem}
\begin{proof}
This easy computation goes as follows. By the change of variable
$g=\gamma^{-1}$ and the symmetry of $\mu$, we have
  \[
  \int f(\gamma,\xi) \dd \mu(\gamma) \dd \nu(\xi) = \int f(g^{-1},\xi) \dd \mu(g)\dd \nu(\xi).
  \]
The cocycle relation $f(gg',\xi)=f(g,g'\xi)+f(g',\xi)$ implies that
$f(g^{-1},\xi)=-f(g,g^{-1}\xi)$. The change of variable
$\eta=g^{-1}\xi$ gives
  \[
  \int f \dd \mu \dd \nu = - \int  f(g,\eta) \dd(g^{-1}_*\nu)(\eta) \dd \mu(g)= \int  f \frac {-2c}{1+c}\dd m.
  \]
On the other hand, we have of course
  \[
  \int f \dd \mu \dd \nu = \int f \frac{2}{1+c}\dd m.
  \]
The half-sum of these two relations gives the desired result.
\end{proof}

\subsection{The Poisson boundary, proof of the first main inequality}

In this paragraph, we prove the first inequality~\eqref{eq_main_rho}
of our main theorem, relating $\rho$ and $h$. The proof relies on the
action of $(\Gamma,\mu)$ on its \emph{Poisson boundary}
$(\mathcal{B}_0,\nu_0)$, that we described quickly in
Section~\ref{sec:examples} (see~\cite{kaim_vershik, Fur02} for more
details). Let $c_0$, $d_0$ and $m_0$ be the objects defined above,
attached to the Poisson boundary.

Kaimanovich and Vershik~\cite{kaim_vershik} proved the following
formula for the entropy:
  \begin{equation}
  \label{eq:entropy_Poisson}
  h = -\int_{\Gamma\times \mathcal{B}_0 } \log c_0 \dd \mu \dd \nu_0.
  \end{equation}
Since the Radon-Nikodym derivative $c_0$ is a multiplicative cocycle,
the symmetrization lemma~\ref{lemsym} applies:
  \[
  h= - \int_{\Gamma\times \mathcal{B}_0 } \log c_0 \cdot d_0 \dd m_0.
  \]
We have $c_0 = \frac{1-d_0}{1+d_0}$ by definition of $d_0$. Hence,
writing
  \begin{equation}
  \label{eq:def_F}
  F(x) = 2x \argth(x) = x \log\left(\frac{1+x}{1-x} \right) =
  \sum_{n\geq 1} \frac{2}{2n-1}x^{2n},
  \end{equation}
we get the following expression for the entropy:
  \begin{equation}
  \label{entropy_Poisson}
  h= \int_{\Gamma\times \mathcal{B}_0 } F(d_0) \dd m_0.
  \end{equation}
Note that the function $F$ is even.

We will now get a bound from below for the spectral radius of the
random walk, using an object living on the Poisson boundary.

\begin{lem}\label{rho-c}
One has
  \begin{equation}
  \label{wpxoucvmkljxwcv}
  \rho \geq \int_{\Gamma \times \mathcal{B}_0 } c_0^{1/2} \dd \mu \dd \nu_0.
  \end{equation}
\end{lem}
\begin{arxivversion}
We give two proofs of this lemma, an elementary one that is mainly
done inside the group, and a more conceptual one written directly on
the boundary.
\begin{proof}[First proof, inside the group]
\end{arxivversion}
\begin{submittedversion}
\begin{proof}
\end{submittedversion}
Let us define a function $f_n$ on $\Gamma$ by $f_n(x) = \mu^{*
n}(x)^{1/2}$. It has unit norm in $L^2$. Denoting by $P$ the Markov
operator associated to the random walk, we get:
  \begin{equation}
  \label{eq_rho_below}
  \begin{split}
  \rho &\geq \langle Pf_{n-1}, f_n \rangle
  =\sum_x\sum_g \mu(g) f_{n-1}(xg) f_n(x)
  =\sum_y \sum_g \mu(g) f_{n-1}(g^{-1}y) f_n(y)
  \\&
  =\sum_y \sum_g \mu(g) \left(\frac{ \mu^{*n-1}(g^{-1}y)}{\mu^{*n}(y)}\right)^{1/2} \mu^{*n}(y).
  \end{split}
  \end{equation}
Let $\PP$ denote the probability distribution of the random walk on
the space $\Omega$ of trajectories starting from the identity. Write
$\omega_n$ for the position at time $n$ of a trajectory, $\boF_n$ for
the $\sigma$-algebra generated by $\omega_n, \omega_{n+1},\dotsc$.
and $C_g$ for the set of trajectories with $\omega_1=g$. Then
  \begin{equation*}
  \frac{\mu(g) \mu^{*n-1}(g^{-1}\omega_n)}{\mu^{*n}(\omega_n)}
  =\PP(C_g \given \boF_n)(\omega).
  \end{equation*}
This converges almost surely to $\PP(C_g \given
\boF_\infty)(\omega)$, where $\boF_\infty =\bigcap \boF_n$ is the
tail $\sigma$-algebra. The Poisson boundary $(\mathcal{B}_0, \nu_0)$
is the quotient of $(\Omega, \PP)$ by $\boF_\infty$. Denoting by
$\bnd:\Omega\to \mathcal{B}_0$ the quotient map, we deduce that
$\mu^{*n-1}(g^{-1}\omega_n)/\mu^{*n}(\omega_n)$ converges almost
surely to a function of $\bnd(\omega)$, which is in fact $c_0(g,
\bnd(\omega))$ (see~\cite[Paragraph 3.2]{Kaim00}). This function is
bounded from above (by $\mu(g)^{-1}$), hence convergence in $L^1$
follows. We obtain
  \begin{align*}
  \sum_y \left(\frac{ \mu^{*n-1}(g^{-1}y)}{\mu^{*n}(y)}\right)^{1/2} \mu^{*n}(y)
  &=\int_\Omega \bigl(\mu(g)^{-1}\PP(C_g \given \boF_n)(\omega)\bigr)^{1/2} \dd \PP(\omega)
  \\&
  \to \int_\Omega c_0(g,\bnd(\omega))^{1/2} \dd\PP(\omega)
  = \int_{\mathcal{B}_0} c_0(g, \xi)^{1/2} \dd\nu_0(\xi).
  \end{align*}
The result follows from this convergence and~\eqref{eq_rho_below}.
\end{proof}

\begin{submittedversion}
\begin{rmk}
A ``more abstract'' proof of Lemma ~\ref{rho-c} can also be obtained
by using the fact that since the action of the group $\Gamma$ on the
Poisson boundary $(\mathcal{B}_0,\nu_0)$ is amenable ~\cite[Cor
5.3]{Zim78}, the associated quasi-regular representation $\pi$ of
$\Gamma$ is weakly contained in the regular representation $\pireg$
~\cite{Kuh94}. Denote by $\pireg(\mu)$ and $\pi(\mu)$ the averages of
$\pireg$ and $\pi$ with respect to $\mu$. Now, the right hand side of
inequality ~\ref{wpxoucvmkljxwcv} is precisely $\langle
\pi(\mu)1,1\rangle$, whereas by the above weak containment property
$\|\pi(\mu)\|\le \|\pireg(\mu)\|=\rho$.
\end{rmk}
\end{submittedversion}

\begin{arxivversion}
\begin{proof}[Second proof, on the boundary]
By a theorem of Zimmer~\cite[Cor 5.3]{Zim78}, the -- ergodic --
action of the discrete group $\Gamma$ on its Poisson boundary is
amenable in the sense of Zimmer. The precise definition of this
notion will not be important for us, we will only use the following
consequence.

Consider the following two unitary representations of $\Gamma$: the
regular representation $\pireg$ defined on $\ell^2(\Gamma)$ by
\[
(\pireg(\gamma)f)(g)=f(\gamma^{-1} g)
\]
and the representation $\pi$ defined on $L^2(\mathcal{B}_0,\nu_0)$ by
  \[
  (\pi(\gamma)f)(\xi)=c_0(\gamma^{-1},\xi)^{1/2}f(\gamma^{-1} \xi).
  \]
Denote by $\pireg(\mu)$ and $\pi(\mu)$ the averages of $\pireg$ and
$\pi$ with respect to $\mu$, namely:
\[
\pireg(\mu)=\sum_{\gamma\in\Gamma}\mu(\gamma)\pireg(\gamma)
        \quad \text{and} \quad \pi(\mu)=\sum_{\gamma\in\Gamma}\mu(\gamma)\pi(\gamma).
\]
Since the representations $\pireg$ and $\pi$ are unitary and the
measure $\mu$ is symmetric, the operators $\pireg(\mu)$ and
$\pi(\mu)$ are self-adjoint. The operator $\pireg(\mu)$ is just the
Markov operator $P$ associated to the random walk on $\Gamma$.

A theorem of Kuhn~\cite{Kuh94} (valid for ergodic amenable actions)
implies that the representation $\pi$ is weakly contained in the
regular representation $\pireg$. We deduce that the operator
$\pi(\mu)$ has norm less than or equal to the norm of $\pireg(\mu)$,
which is, by a result of Kesten~\cite{Kes59}, exactly the spectral
radius $\rho$. If we consider the scalar product $\langle \pi(\mu)1,1
\rangle$ in $L^2(\mathcal{B}_0,\nu_0)$, we have:
  \[
  \rho = \norm{\pireg(\mu)} \geq \norm{\pi(\mu)} \geq \langle \pi(\mu)1,1 \rangle
                         =\int_{\Gamma \times \mathcal{B}_0 } c_0(\gamma^{-1},\xi)^{1/2} \dd \mu \dd \nu_0
                         =\int_{\Gamma \times \mathcal{B}_0 } c_0^{1/2} \dd \mu \dd \nu_0
  \]
again since the measure $\mu$ is symmetric. This is the desired
result.
\end{proof}
\end{arxivversion}

Since $\dd m_0 = \frac{c_0+1}{2} \dd\mu\dd\nu_0$, the integral
in~\eqref{wpxoucvmkljxwcv} is equal to
  $
  \int_{\Gamma \times \mathcal{B}_0 } \frac{2c_0^{1/2}}{1+c_0} \dd m_0
  $.
We rewrite this expression in terms of $d_0$: since
$c_0=(1-d_0)/(1+d_0)$, we have
  \begin{equation*}
  2c_0^{1/2}\cdot \frac{1}{1+c_0} = 2 \left(\frac{1-d_0}{1+d_0}\right)^{1/2} \cdot \frac{1}{2/(1+d_0)}
  =(1-d_0^2)^{1/2}.
  \end{equation*}
Therefore,
  \begin{equation}
  \label{eq:ineq_rho}
  1-\rho \leq 1 - \int_{\Gamma \times \mathcal{B}_0 } \frac{2c_0^{1/2}}{1+c_0}\dd m_0
  = \int_{\Gamma \times \mathcal{B}_0 } G(d_0) \dd m_0,
  \end{equation}
where $G(x) = 1- (1-x^2)^{1/2}$. This function is even on $[-1,1]$,
its restriction to $[0,1]$ is an increasing bijection of $[0,1]$.

\begin{lem}
\label{lem_etude_FGmoinsun}
The function $F\circ G^{-1}$ satisfies on $[0,1)$
  \begin{equation*}
  F\circ G^{-1}(x) = (2x-x^2)^{1/2} \log\left(\frac{1+(2x-x^2)^{1/2}}{1-(2x-x^2)^{1/2}}\right)
  =\sum_{n=1}^\infty c_n x^n,
  \end{equation*}
where $c_1 = 4$ and $(2n-1) c_n = (n-2)c_{n-1}+2$ for $n\geq 2$. In
particular, the coefficients $c_n$ are positive. Hence, $F\circ
G^{-1}$ is increasing and convex.
\end{lem}
\begin{proof}
A simple computation shows that the function $H=F\circ G^{-1}$
satisfies the differential equation
  \begin{equation*}
  H'(x) = \frac{1-x}{x(2-x)}H(x) + \frac{2}{1-x}.
  \end{equation*}
Multiplying by $x(2-x)$ and identifying the Taylor coefficients on
the left and on the right, one gets the recurrence relation $(2n-1)
c_n = (n-2)c_{n-1}+2$ for $n\geq 2$.
\end{proof}

The map $F\circ G^{-1}$ is increasing, so the
inequality~\eqref{eq:ineq_rho} transforms into
  \begin{equation}
  \label{eq_jensen}
  F\circ G^{-1}(1- \rho) \leq
  F\circ G^{-1} \left( \int_{\Gamma \times \mathcal{B}_0 } G(d_0) \dd m_0 \right).
  \end{equation}
Note that the partial inverse $G^{-1}$ of $G$ satisfies $F\circ
G^{-1} \circ G=F$ on the interval $(-1,1)$, because both $F$ and $G$
are even functions. Since $F\circ G^{-1}$ is convex by
Lemma~\ref{lem_etude_FGmoinsun}, Jensen's inequality implies that
  \[
  F\circ G^{-1}(1- \rho) \leq \int_{\Gamma \times \mathcal{B}_0 } F(d_0) \dd m_0.
  \]
Thanks to~\eqref{entropy_Poisson}, this proves~\eqref{eq_main_rho}
since $G^{-1}(1-\rho)=\sqrt{1-\rho^2}$. \qed

\subsection{The Busemann compactification, proof of the second main inequality}
\label{subsec_horocomp}

For the proof of the second inequality~\eqref{eq:main_ell} of our
main theorem, relating $\ell$ and $h$, we will need another more
geometric boundary, which will give us access to the metric notion of
linear drift, in contrast to the Poisson boundary which is purely a
measure theoretic construction.

We recall the construction of the Busemann (horospherical) closure of
the group $\Gamma$. It is obtained by embedding $\Gamma$ into
Lipschitz functions on $\Gamma$ using the distance kernel, as
follows. Let $X \subset \R^\Gamma$ be the set of $1$-Lipschitz
real-valued functions on $\Gamma$ which vanish on $e$. Lipschitz
means here that $\abs{\phi(gg')-\phi(g)}\leq \lgth{g'}$. For any
$\gamma \in \Gamma$,
  \[
  \Phi_\gamma(g)=\lgth{\gamma^{-1}g}-\lgth{\gamma^{-1}}
  \]
defines an element of $X$, and the assignment $\gamma \mapsto
\Phi_\gamma$ is continuous, injective. Let $\mathcal{B}_1$ be the
closure of the image of $\Gamma$. The action of $\Gamma$ on
$\mathcal{B}_1$ is given by
  \[
  (\gamma \xi)(g)=\xi(\gamma^{-1}g)-\xi(\gamma^{-1}).
  \]
The latter equation for the action is better understood if one thinks
of $X$ as the quotient set of $1$-Lipschitz functions on $\Gamma$
modulo the constants, endowed with the natural translation action on
functions. Each element of $X$ has a unique representative which
vanishes at $e$, which explains the above formula.

\medskip

Karlsson and Ledrappier~\cite{KaLe07}, \cite[Thm 18]{KaLe11} proved
that in this setting, under the assumption of finite first moment,
there exists an ergodic stationary probability measure $\nu_1$ on
$\mathcal{B}_1$ satisfying:
  \[
  \ell=\int_{\Gamma \times \mathcal{B}_1} \xi(\gamma^{-1}) \dd \mu(\gamma) \dd \nu_1(\xi).
  \]
\cite{KaLe11} calls this expression for $\ell$ a
Furstenberg-Khasminskii formula.

By definition of the action, the assignment $\beta: (g,\xi) \mapsto
\xi(g^{-1})$ satisfies
  \[
  \beta(gg',\xi)=\xi(g'^{-1}g^{-1})=(g'\xi)(g^{-1})+ \xi({g'}^{-1})=\beta(g,g'\xi)+\beta(g',\xi),
  \]
so it is an additive cocycle; this is in fact the classical Busemann
cocycle. Hence, the symmetrization lemma~\ref{lemsym} applies, and we
find
  \begin{equation}
  \label{formula_for_ell}
  \ell=\int_{\Gamma \times \mathcal{B}_1} \beta \cdot d_1 \dd m_1,
  \end{equation}
where $d_1=(1-c_1)/(1+c_1)$ with $c_1$ the Radon-Nikodym derivative,
and $m_1=\frac{1+c_1}2 \dd \mu \dd \nu_1$.

\medskip

Kaimanovich and Vershik~\cite{kaim_vershik} proved that the boundary
entropy of a $(\Gamma,\mu)$-space is always less than or equal to the
entropy of the random walk. Applying this result to $(\mathcal{B}_1,
\nu_1)$, we get
  $
  - \int_{\Gamma \times \mathcal{B}_1} \log c_1 \dd \mu \dd \nu_1 \leq h
  $.
The left hand side can be transformed using the symmetrization
lemma~\ref{lemsym}, giving
  \begin{equation}
  \label{eq:ineq_h_B1}
  \int_{\Gamma\times \mathcal{B}_1 } F(d_1) \dd m_1 \leq h.
  \end{equation}

\medskip

We can now prove our second main inequality~\eqref{eq:main_ell}
comparing $\ell$ and $h$. We start from~\eqref{formula_for_ell} and
apply Cauchy-Schwarz inequality, yielding
  \begin{equation}
  \label{eq:after_CS}
  \ell \leq \left(\int_{\Gamma\times \mathcal{B}_1 } \abs{\beta}^2 \dd m_1\right)^{1/2}
  \left( \int_{\Gamma\times \mathcal{B}_1 } d_1^2 \dd m_1 \right)^{1/2}.
  \end{equation}
Since $\abs{\beta(g,\xi)}\leq \lgth{g}$, because $\mathcal{B}_1$
consists of 1-Lipschitz functions vanishing at $e$, the first factor
on the right hand side is bounded by $M_2(\mu)$. Writing $\tilde\ell
= \ell/M_2(\mu)$, we obtain
  \begin{equation*}
  \tilde\ell^2 \leq \int_{\Gamma\times \mathcal{B}_1 } d_1^2 \dd m_1.
  \end{equation*}
It follows from the Taylor expansion of the function $F$, given
in~\eqref{eq:def_F}, that $\tilde F(x) = F(x^{1/2})$ is convex on
$[0,1)$. Applying $\tilde F$ to the previous inequality and using
Jensen inequality, we get
  \begin{equation*}
  F(\tilde \ell) \leq \int_{\Gamma\times \mathcal{B}_1 } F(d_1) \dd m_1.
  \end{equation*}
By~\eqref{eq:ineq_h_B1}, the right hand side is bounded by $h$. This
proves~\eqref{eq:main_ell}. \qed

\medskip

The above proof can be refined, to get a slightly stronger
inequality. For any $p\geq 1$, let $M_p(\mu) = \left(\sum_g
\lgth{g}^p \mu(p)\right)^{1/p}$ be the $\ell^p$-norm with respect to
the measure $\mu$ of the distance to the identity, generalizing the
notation $M_2(\mu)$.

\begin{prop}
\label{prop_stronger}
Let $\mu$ be a symmetric probability measure with finite first moment
on a countable group $\Gamma$ with a proper left-invariant distance.
Then
\begin{equation*}
   \sum_{n=1}^\infty \frac{2}{2n-1} \left(\frac{\ell}{M_{1+1/(2n-1)}(\mu)}\right)^{2n}
  \leq h.
  \end{equation*}
\end{prop}

In this estimate, the first terms of the expansion vanish if the
corresponding moments are infinite. This proposition gives a
nontrivial estimate when $\mu$ has a finite moment of some order
$p>1$. In particular, if $\mu$ has a moment of order $1+1/(2n-1)$, we
get
  \begin{equation*}
  \ell \leq M_{1+1/(2n-1)}(\mu) \left(\frac{2n-1}{2} h\right)^{1/(2n)}.
  \end{equation*}
If $h=0$ for such a measure, it follows that $\ell=0$. This is a weak
version of a theorem of Karlsson and Ledrappier~\cite{KaLe07},
stating that this implication holds for symmetric measures with a
finite moment of order $1$ (the symmetry assumption can even be
replaced by a weaker centering assumption).

Note that, since $M_p(\mu) \leq M_2(\mu)$ for $p\leq 2$ and $F(x) =
\sum\frac{2}{2n-1}x^{2n}$, this proposition strengthens the
inequality~\eqref{eq:main_ell}.

\begin{proof}[Proof of Proposition~\ref{prop_stronger}]
Let $n\geq 1$ be an integer. We start again
from~\eqref{formula_for_ell}, but we use H\"older inequality for the
exponent $1+1/(2n-1)$ and the conjugate exponent $2n$: it follows
that the drift $\ell$ satisfies
  \begin{equation*}
  \ell \leq
  \left( \int_{\Gamma \times \mathcal{B}_1} \abs{\beta}^{1+1/(2n-1)} \dd m_1 \right)^{(2n-1)/(2n)}
  \left( \int_{\Gamma \times \mathcal{B}_1} d_1^{2n} \dd m_1 \right)^{1/(2n)}.
  \end{equation*}
Since $\abs{\beta(g,\xi)}\leq \lgth{g}$, the first factor is bounded
by $M_{1+1/(2n-1)}(\mu)$. Thus,
  \[
  \left( \frac{\ell}{M_{1+1/(2n-1)}(\mu)} \right)^{2n}
  \leq \int_{\Gamma \times \mathcal{B}_1} d_1^{2n} \dd m_1.
  \]
Note that the previous equation makes sense even if $\mu$ has no
finite moment of order $1+1/(2n-1)$ (in this case, the left hand size
vanishes, and the equation is trivial).

Multiplying this inequality by $2/(2n-1)$ and summing over $n$, we
obtain
  \begin{equation*}
  \sum_{n\geq 1} \frac{2}{2n-1} \left( \frac{\ell}{M_{1+1/(2n-1)}(\mu)} \right)^{2n}
  \leq \int_ {\Gamma \times \mathcal{B}_1} \sum_{n\geq 1} \frac{2}{2n-1}d_1^{2n} \dd m_1
  = \int_ {\Gamma \times \mathcal{B}_1} F(d_1) \dd m_1.
  \end{equation*}
By~\eqref{eq:ineq_h_B1}, this is at most $h$.
\end{proof}

\subsection{Discussion of the equality case}
\label{subsec_equality}

The proofs given in the previous paragraphs imply that the equality
situation in those inequalities is very rigid. We can use this
information to prove Proposition~\ref{equal}.

\begin{proof}[Proof of Proposition~\ref{equal}]
Assume first that the inequality~\eqref{eq_main_rho} comparing $\rho$
and $h$ is an equality. Then all the inequalities in the proof of
this inequality have to be equalities. In particular, Jensen's
inequality after~\eqref{eq_jensen} is an equality, whence $G(d_0)$ is
almost surely constant, i.e., there exists $a\in \R$ such that
$d_0=\pm a$ almost surely. Since $c_0=(1-d_0)/(1+d_0)$, it follows
that $c_0$ almost surely takes the values $(1-a)/(1+a)$ or
$(1+a)/(1-a)$, which are inverse of each other.

Assume now that the inequality~\eqref{eq:main_ell} comparing $\ell$
and $h$ is an equality. Denote by $(\mathcal{B}_1, \nu_1)$ the
Busemann compactification used in Paragraph~\ref{subsec_horocomp}. We
have the inequalities
  \begin{multline}
  \label{eq_multline}
  F\left(\frac{\ell}{M_2(\mu)}\right) = \sum \frac{2}{2n-1} \left(\frac{\ell}{M_2(\mu)}\right)^{2n}
  \leq \sum \frac{2}{2n-1} \left(\frac{\ell}{M_{1+1/(2n-1)}(\mu)}\right)^{2n}
  \\
  \leq \int_{\Gamma\times \mathcal{B}_1} F(d_1) \dd m_1
  = -\int_{\Gamma \times \mathcal{B}_1}\log c_1 \dd \mu \dd\nu_1
  \leq h.
  \end{multline}
If the extreme terms are equal, we have equality everywhere.

All the moments of $\mu$ coincide, hence $\mu$ is supported on points
at a fixed distance of $e$. There must also be equality $m_1$-almost
everywhere in the inequality $\abs{\beta(g,\xi)}\leq \lgth{g}$ that
we used just after~\eqref{eq:after_CS}. This implies that
$\abs{\beta}$ is almost surely constant. Finally, there is equality
in the Cauchy-Schwarz inequality~\eqref{eq:after_CS}, hence, $d_1$ is
almost surely proportional to $\beta$. It follows that $\abs{d_1}$ is
almost surely constant. Hence, as in the first case, $c_1$ takes only
two values which are inverse of each other. To conclude, we should
prove that this property (that we have proved on
$(\mathcal{B}_1,\nu_1)$) also holds on the Poisson boundary, since
the statement of the proposition is formulated on the Poisson
boundary.

Since equality holds everywhere in~\eqref{eq_multline}, one has in
particular $-\int _{\Gamma\times \mathcal{B}_1} \log c_1 \dd \mu
\dd\nu_1= h$, i.e., the entropy of the $(\Gamma,\mu)$-space
$(\mathcal{B}_1, \nu_1)$ is maximal. By~\cite[Theorem
3.2]{kaim_vershik}, this implies that $(\mathcal{B}_1, \nu_1)$ is the
Poisson boundary if the Radon-Nikodym cocycle separates the points,
i.e., if for almost every points $\xi\not=\eta$ there exists $g\in
\Gamma$ such that $c_1(g,\xi)\not= c_1(g,\eta)$. In general, the
Poisson boundary is a factor of $(\mathcal{B}_1,\nu_1)$, obtained by
identifying the points that are not separated by the Radon-Nikodym
cocycle. In particular, any property of the Radon-Nikodym cocycle
that is true on $(\mathcal{B}_1,\nu_1)$ is also true on the Poisson
boundary. This concludes the proof.
\end{proof}

\begin{arxivversion}

\section{Lower bounds involving both the spectral radius and the drift}
\label{sec:chebyshev}

Combining both inequalities in Theorem~\ref{main_thm}, one can obtain
other inequalities involving at the same time $\rho$, $\ell$ and $h$,
including notably the following corollary.
\begin{cor}
\label{cor_chebyshev}
Let $A(x) = (1+x)\log(1+x)+(1-x)\log(1-x)$. Then any symmetric random
walk with finite second moment on a countable group satisfies
  \begin{equation*}
  A(\ell/M_2(\mu)) + 2\abs{\log \rho}\leq h.
  \end{equation*}
\end{cor}
\begin{proof}
Let $F(x)=2x\argth(x)$, Theorem~\ref{main_thm} gives
$F(\sqrt{1-\rho^2}) \leq h$ and $F(\tilde \ell) \leq h$ with
$\tilde\ell=\ell/M_2(\mu)$. Writing $s=F^{-1}(h)$, one has
$\sqrt{1-\rho^2} \leq s$, hence $\rho^2 \geq 1-s^2$, hence $-2\log
\rho \leq -\log(1-s^2)$. Since $A$ is increasing, we obtain
  \begin{align*}
  A(\tilde\ell) -2\log \rho
  &\leq A(s) -\log(1-s^2)
  \\&
  =(1+s) \log(1+s)+(1-s)\log(1-s) - \log(1+s) -\log(1-s)
  \\&
  = s \log\left(\frac{1+s}{1-s}\right)
  = F(s)=h.
  \qedhere
  \end{align*}
\end{proof}

Note that this statement improves both Avez and Carne-L{\oe}uillot
inequalities that we explained in the introduction. Surprisingly, for
nearest neighbor random walks, we found a direct (and completely
different) proof of this result, relying on properties of Chebyshev
polynomials and on large deviation estimates for the simple random
walk on $\Z$, inspired by the techniques of Carne~\cite{Carne85}. The
function $A$ appears naturally in this proof as a large deviations
rate. Since the argument is interesting in its own right, we will
explain it in the rest of this section.

\medskip

We will always write $A$ for the function in the statement of the
corollary. Let $\Gamma$ be a countable group with a proper
left-invariant distance, and let $\mu$ be a symmetric measure
supported on $B(e,1)$, we want to prove that it satisfies $A(\ell) +
2 \abs{\log\rho}\leq h$.

Consider the Hilbert space $\ell^2(\Gamma)$, with its scalar product
$\langle \cdot ,\cdot \rangle$. The Markov operator $P_{\mu}$
associated to $\mu$ is defined by $P_{\mu}f(g) = \sum_{h\in
\Gamma}\mu(h)f(gh)$. It is a contraction on $\ell^2(\Gamma)$, and its
iterates are given by $P^n_{\mu} = P_{\mu^{*n}}$.

Since the measure $\mu$ is symmetric, the operator $P_{\mu}$ is
self-adjoint, therefore its spectrum $\sigma (P_{\mu})$ is real and
contained in the interval $[-1,1]$. Moreover, the spectral radius of
$P_{\mu}$ is given by
\[
\rho (P_{\mu}) = \sup_{\lambda \in \sigma (P_{\mu})}\abs{\lambda} = \norm{P_{\mu}} = \rho.
\]
The second equality holds for every self-adjoint operator.
See~\cite{Kes59} for the last equality.

\medskip

If $K \subset \Gamma$, we write $\I_K$ for the indicator function of
$K$. It belongs to $\ell^2(\Gamma)$ when $K$ is finite. If $K
=\{g\}$, we simply write $\I_g$. We have, for every $g\in \Gamma$,
  \begin{equation*}
  \langle P_{\mu}^n \I_e, \I_g \rangle
  = (P_{\mu}^n \I_e) (g)
  = \sum_{h\in \Gamma}\mu^{*n}(h) \I_e (gh)
  = \mu^{*n}(g^{-1}) = \mu^{*n}(g)
  \end{equation*}
since $\mu$ is symmetric. More generally, for every $K \subset
\Gamma$,
  \begin{equation*}
  \langle P_{\mu}^n \I_e,\I_K \rangle = \mu^{*n}(K).
  \end{equation*}

\begin{lem}\label{KL} Let $(T_k(X))_k$ be the sequence of Chebyshev polynomials
and let $(S_n)_n$ be the simple random walk on $\Z$. Then:
  \begin{enumerate}
    \item for every $n\in \N$, one has $
        X^n=\sum_{k=0}^{n}\PP(\abs{S_n}=k)T_k(X)$;
    \item for every self-adjoint operator $u$ of a Hilbert space
        with unit norm, $\norm{T_k(u)}=1$ for all $k\in\N$;
    \item for every $k,n\in \N$ such that $0\leq k\leq n$, one
        has $\PP(S_n \geq k)\leq\exp\bigl[-\frac{n}{2}
        A\bigl(k/n\bigr)\bigr]$.
  \end{enumerate}
\end{lem}
\begin{proof}
(1) This is~\cite[Thm.~2]{Carne85}. We recall Carne's proof in order
to be complete. Set $x = \cos t $. Then:
\begin{align*}
x^n & =  \frac{1}{2^n}\bigl(e^{it}+e^{-it}\bigr)^n=\sum_{k=-n}^{n}\PP(S_n=k)e^{ikt}
      =  \sum_{k=0}^{n}\PP(\abs{S_n}=k)\frac{e^{ikt}+e^{-ikt}}{2}\\
    & =  \sum_{k=0}^{n}\PP(\abs{S_n}=k) \cos kt
      =  \sum_{k=0}^{n}\PP(\abs{S_n}=k) T_k(\cos t)
      =  \sum_{k=0}^{n}\PP(\abs{S_n}=k) T_k(x) .
\end{align*}
(2) The Chebyshev polynomials satisfy $T_k([-1,1]) = [-1,1]$ and
$\abs{T_k(\pm 1)}=1$. Moreover, since $T_k$ is real and $u$
self-adjoint, the operator $T_k(u)$ is also self-adjoint. If
$\norm{u}=1$, then $\sigma(u) \subset [-1,1]$, hence $\sigma (T_k(u))
= T_k(\sigma(u)) \subset [-1,1]$. We have
\[
\norm{T_k(u)} = \sup\{\abs{\lambda},\lambda \in \sigma (T_k(u))\}
              = \sup\{\abs{T_k(\lambda)},\lambda \in \sigma (u)\}=1.
\]
(3) This is a standard Chernov type estimate. For every real $t > 0$,
we have, using Markov inequality,
\[
\PP(S_n \geq k) = \PP(e^{tS_n} \geq e^{tk})
                \leq e^{-tk}\E(e^{tS_n})
                = e^{-tk}(\cosh t)^n = \exp \biggl[-n\biggl(t\frac{k}{n} - \log\cosh t\biggr)\biggr] .
\]
An elementary computation shows that, for $x\in [0,1]$,
$\sup\{tx -\log\cosh t,t > 0\}=A(x)/2$. The result follows. Observe that the
function $A$ appears as twice the Legendre transform of the function
$\log\cosh$, hence is convex.
\end{proof}

Recall that, for every $K \subset \Gamma$, we have
$\mu^{*n}(K)=\langle P_{\mu}^n \I_e, \I_K \rangle$. Applying Item (1)
of lemma~\ref{KL}, we get
\begin{equation}\label{decouplage}
\frac{1}{\rho^n}\mu^{*n}(K)
  = \left\langle \left(\frac{1}{\rho}P_{\mu}\right)^n \I_e, \I_K \right\rangle
  = \sum_{k=0}^{n}\PP(\abs{S_n}=k)\left\langle
           T_k\left(\frac{1}{\rho}P_{\mu}\right) \I_e, \I_K \right\rangle.
\end{equation}

What remains to do is to apply this formula to a suitable sequence of
finite subsets of $\Gamma$. Fix $\epsilon > 0$. Let $K_n \subset
\Gamma$ be defined by
\[
K_n=\{g\in \Gamma \st \abs{g} \in [\ell(1-\epsilon)n,\ell(1+\epsilon)n]
      \text{ and } \mu^{*n}(g) \in [e^{-h(1+\epsilon)n},e^{-h(1-\epsilon)n}]\}.
\]
Recall that we denote by $(X_n)_n$ (a realization of) the right
random walk associated with $(\Gamma,\mu)$. Using Kingman's
subadditive ergodic theorem~\cite{Derrien80}, one can prove that, as
$n\rightarrow +\infty$, $\abs{X_n}/n\rightarrow \ell$ and
$-\log\mu^{*n}(X_n)/n\rightarrow h$ almost surely, hence in
probability. Therefore, $\lim_n \mu^{*n}(K_n) = 1$. In particular,
taking $n$ large enough, one has $\mu^{*n}(K_n) \geq 1-\epsilon$.

\medskip

Denote by $\#K$ the cardinality of a set $K\subset \Gamma$. We have
\[
1 \geq \mu^{*n}(K_n) \geq \#K_ne^{-h(1+\epsilon)n},
\]
hence $\#K_n \leq e^{h(1+\epsilon)n}$.

\medskip

Observe that, since $\deg T_k = k$ and $\supp (\mu) \subset B(e,1)$,
the support of the function $T_k\left(\frac{1}{\rho}P_{\mu}\right)
\I_e$ is contained in the ball $B(e,k)$, and therefore is disjoint
from the support of the function $\I_{K_n}$ if $k <
\ell(1-\epsilon)n$. The identity~\eqref{decouplage} written with the
set $K_n$ then becomes
\begin{equation*}
\frac{1}{\rho^n}\mu^{*n}(K_n)=
  \sum_{\ell(1-\epsilon)n \leq k \leq n}\PP(\abs{S_n}=k)
          \left\langle T_k\left(\frac{1}{\rho}P_{\mu}\right) \I_e,\I_{K_n} \right\rangle.
\end{equation*}
Using Cauchy-Schwarz inequality and the item (2) of lemma~\ref{KL},
we obtain
\begin{align*}
\frac{1}{\rho^n}\mu^{*n}(K_n)
               & \leq \sum_{\ell(1-\epsilon)n \leq k \leq n}
                       \PP(\abs{S_n} = k)\norm{T_k\left(\frac{1}{\rho}
                       P_{\mu}\right)}\cdot\norm{\I_e}_2 \cdot \norm{\I_{K_n}}_2 \\
               & \leq  \PP(\abs{S_n} \geq \ell(1-\epsilon)n)\sqrt{\# K_n} \\
               & \leq  2\PP(S_n \geq \ell(1-\epsilon)n)\sqrt{e^{h(1+\epsilon)n}} .
\end{align*}
The item (3) of lemma~\ref{KL} yields
\[
\frac{1-\epsilon}{\rho^n} \leq \frac{1}{\rho^n}\mu^{*n}(K_n)
\leq 2\exp\bigg[-\frac{n}{2}A(\ell(1-\epsilon))+\frac{1}{2}h(1+\epsilon)n\bigg] .
\]
Taking the logarithm of both sides and dividing by $n$ gives
\[
\frac{\log(1-\epsilon)}{n} - \log\rho
        \leq \frac{\log 2}{n} -\frac{1}{2}A(\ell(1-\epsilon))+\frac{1}{2}h(1+\epsilon)
        .
\]
Letting $n\rightarrow +\infty$ and $\epsilon \rightarrow 0$ we get $-
\log\rho \leq -A(\ell)/2+h/2$. This concludes the direct proof of
Corollary~\ref{cor_chebyshev} for nearest-neighbor random walks. \qed

\begin{rmk}
Let $\mu$ be a symmetric probability measure on $\Gamma$ supported in
$B(e,1)$. Writing Equation~\eqref{decouplage} for $K=\{g\}$ and
following the above proof leads to
\[
\frac{1}{\rho^n}\mu^{*n}(g) =\sum_{\lgth{g}\leq k\leq n}\PP(\abs{S_n}=k)
          \biggl\langle T_k\biggl(\frac{1}{\rho}P_{\mu}\biggr) \I_e,\I_g \biggr\rangle
          \leq \PP(\abs{S_n}\geq \lgth{g}).
\]

Therefore we have
\[
\mu^{*n}(g)\leq 2\rho^n \PP(S_n\geq \lgth{g})
\leq 2\rho^n\exp \biggl[-\frac{n}{2}A\biggl(\frac{\lgth{g}}{n}\biggr)\biggr],
\]
which is L{\oe}uillot's upper bound for $\mu^{*n}(g)$ (see~\cite{Loe11}).
It is also possible to get \emph{lower} bounds for $\mu^{*n}(g)$
using lemma~\ref{KL}, see~\cite[Thm.~14.22]{Woe00}.
\end{rmk}

We can deduce other inequalities from Theorem~\ref{main_thm}, for
instance the following corollary (strengthening Ledrappier's
inequality $4(1-\rho)\leq h$).
\begin{cor}
\label{cor_ledrappier}
Let $A(x) = 2x \argth(x) + 4\sqrt{1-x^2} - 4 \geq 0$. Then any
symmetric random walk with finite second moment satisfies
  \begin{equation}
  \label{eq:cor_ledrappier}
  A(\ell/M_2(\mu)) + 4(1-\rho) \leq h.
  \end{equation}
\end{cor}
\begin{proof}
Let us first show that $A$ is increasing. Using the Taylor expansions
of $\argth(x)$ and $\sqrt{1-x}$, we have
\[
A(x) = 2\sum_{n=0}^{\infty}\biggl(\frac{1}{2n+1}-\frac{1}{4^n(n+1)}\binom{2n}{n}\biggr)x^{2n+2}.
\]
For $n\geq 1$, one may estimate $(2n)!$ by bounding each odd number
in the product by the even number following it. This gives $(2n)!
\leq 4^n (n!)^2$. Bounding only each odd number $>1$ by the even
number following it, we even get $(2n)! \leq 4^n (n!)^2 /2$, hence
$4^{-n} \binom{2n}{n} \leq 1/2$. Therefore all the coefficients of
the Taylor expansion of $A$ are nonnegative, and $A$ is increasing
(and nonnegative).

The proof of the inequality~\eqref{eq:cor_ledrappier} is then
completely similar to the proof of Corollary~\ref{cor_chebyshev}:
setting $B(x)=4x$, then $A(x)+B(1-\sqrt{1-x^2}) = F(x)$, which is the
algebraic property that played a role in this proof, implying that
$A(\tilde\ell)+B(1-\rho)\leq h$.
\end{proof}

\noindent \emph{Acknowledgements:} The authors gladly thank Bachir
Bekka for pointing out to them the references~\cite{Zim78}
and~\cite{Kuh94}.
\end{arxivversion}

\bibliography{biblioLower}

\providecommand{\bysame}{\leavevmode\hbox to3em{\hrulefill}\thinspace}
\providecommand{\MR}{\relax\ifhmode\unskip\space\fi MR }
\providecommand{\MRhref}[2]{%
  \href{http://www.ams.org/mathscinet-getitem?mr=#1}{#2}
}
\providecommand{\href}[2]{#2}
\begin{thebibliography}{BHM11}

\bibitem[Anc88]{Ancona}
Alano Ancona, \emph{Positive harmonic functions and hyperbolicity}, Potential
  theory, surveys and problems ({P}rague, 1987), Lecture Notes in Math., vol.
  1344, Springer, Berlin, 1988, pp.~1--23. \MR{MR973878}

\bibitem[Ave76]{Avez76}
Andr{\'e} Avez, \emph{Croissance des groupes de type fini et fonctions
  harmoniques}, Th\'eorie ergodique ({A}ctes {J}ourn\'ees {E}rgodiques,
  {R}ennes, 1973/1974), Springer, Berlin, 1976, pp.~35--49. Lecture Notes in
  Math., Vol. 532. \MR{MR0482911}

\bibitem[Bar04]{bartholdi_rho}
Laurent Bartholdi, \emph{Cactus trees and lower bounds on the spectral radius
  of vertex-transitive graphs}, Random walks and geometry, Walter de Gruyter
  GmbH \& Co. KG, Berlin, 2004, pp.~349--361. \MR{MR2087788}

\bibitem[BHM11]{BHM}
S{\'e}bastien Blach{\`e}re, Peter Ha{\"{\i}}ssinsky, and Pierre Mathieu,
  \emph{Harmonic measures versus quasiconformal measures for hyperbolic
  groups}, Ann. Sci. \'Ec. Norm. Sup\'er. (4) \textbf{44} (2011), 683--721.
  \MR{MR2919980}

\bibitem[Bow98]{bowditch_2}
Brian~H. Bowditch, \emph{Boundaries of strongly accessible hyperbolic groups},
  The {E}pstein birthday schrift, Geom. Topol. Monogr., vol.~1, Geom. Topol.
  Publ., Coventry, 1998, pp.~51--97. \MR{MR1668331}

\bibitem[Car85]{Carne85}
Thomas~Keith Carne, \emph{A transmutation formula for {M}arkov chains}, Bull.
  Sci. Math. (2) \textbf{109} (1985), 399--405. \MR{MR837740}

\bibitem[Der80]{Derrien80}
Yves Derriennic, \emph{Quelques applications du th\'eor\`eme ergodique
  sous-additif}, Conference on {R}andom {W}alks ({K}leebach, 1979) ({F}rench),
  Ast\'erisque, vol.~74, Soc. Math. France, Paris, 1980, pp.~183--201.
  \MR{MR588163}

\bibitem[Der86]{Der86}
Y.~Derriennic, \emph{Entropie, th\'eor\`emes limite et marches al\'eatoires},
  Probability measures on groups, {VIII} ({O}berwolfach, 1985), Lecture Notes
  in Math., vol. 1210, Springer, Berlin, 1986, pp.~241--284. \MR{MR879010}

\bibitem[dlH00]{laHarpe}
Pierre de~la Harpe, \emph{Topics in geometric group theory}, Chicago Lectures
  in Mathematics, University of Chicago Press, Chicago, IL, 2000.
  \MR{MR1786869}

\bibitem[DM61]{DyMa61}
Evgeni\u\i~B. Dynkin and Mikhail~B. Maljutov, \emph{Random walk on groups with
  a finite number of generators}, Dokl. Akad. Nauk SSSR \textbf{137} (1961),
  1042--1045. \MR{MR0131904}

\bibitem[EK10]{EK12}
Anna Erschler and Anders Karlsson, \emph{Homomorphisms to {$\mathbb{R}$}
  constructed from random walks}, Ann. Inst. Fourier (Grenoble) \textbf{60}
  (2010), 2095--2113. \MR{MR2791651}

\bibitem[Fur02]{Fur02}
Alex Furman, \emph{Random walks on groups and random transformations}, Handbook
  of dynamical systems, {V}ol.\ 1{A}, North-Holland, Amsterdam, 2002,
  pp.~931--1014. \MR{MR1928529}

\bibitem[Gui80]{Guiv80}
Yves Guivarc'h, \emph{Sur la loi des grands nombres et le rayon spectral d'une
  marche al\'eatoire}, Conference on {R}andom {W}alks ({K}leebach, 1979)
  ({F}rench), Ast\'erisque, vol.~74, Soc. Math. France, Paris, 1980,
  pp.~47--98, 3. \MR{MR588157}

\bibitem[Ka{\u\i}86]{Kaim86}
Vadim~A. Ka{\u\i}manovich, \emph{Brownian motion and harmonic functions on
  covering manifolds. {A}n entropic approach}, Dokl. Akad. Nauk SSSR
  \textbf{288} (1986), 1045--1049. \MR{MR852647}

\bibitem[Ka{\u\i}98]{Kaim98}
\bysame, \emph{The {P}oisson formula for groups with hyperbolic properties},
  Preliminary version of the paper in Ann. of Maths. (1998), Preprint
  arXiv:math/9802132v1.

\bibitem[Ka{\u\i}00]{Kaim00}
\bysame, \emph{The {P}oisson formula for groups with hyperbolic properties},
  Ann. of Math. (2) \textbf{152} (2000), 659--692. \MR{MR1815698}

\bibitem[Kes59]{Kes59}
Harry Kesten, \emph{Symmetric random walks on groups}, Trans. Amer. Math. Soc.
  \textbf{92} (1959), 336--354. \MR{MR0109367}

\bibitem[KL07]{KaLe07}
Anders Karlsson and Fran{\c{c}}ois Ledrappier, \emph{Linear drift and {P}oisson
  boundary for random walks}, Pure Appl. Math. Q. \textbf{3} (2007),
  1027--1036. \MR{MR2402595}

\bibitem[KL11]{KaLe11}
\bysame, \emph{Noncommutative ergodic theorems}, Geometry, rigidity, and group
  actions, Chicago Lectures in Math., Univ. Chicago Press, Chicago, IL, 2011,
  pp.~396--418. \MR{MR2807838}

\bibitem[Kuh94]{Kuh94}
M.~Gabriella Kuhn, \emph{Amenable actions and weak containment of certain
  representations of discrete groups}, Proc. Amer. Math. Soc. \textbf{122}
  (1994), 751--757. \MR{MR1209424}

\bibitem[KV83]{kaim_vershik}
Vadim~A. Ka{\u\i}manovich and Anatoly~M. Vershik, \emph{Random walks on
  discrete groups: boundary and entropy}, Ann. Probab. \textbf{11} (1983),
  457--490. \MR{MR704539}

\bibitem[Led90]{ledrappier_continuous}
Fran{\c{c}}ois Ledrappier, \emph{Harmonic measures and {B}owen-{M}argulis
  measures}, Israel J. Math. \textbf{71} (1990), 275--287. \MR{MR1088820}

\bibitem[Led92]{ledrappier_sharp_entropy}
\bysame, \emph{Sharp estimates for the entropy}, Harmonic analysis and discrete
  potential theory ({F}rascati, 1991), Plenum, New York, 1992, pp.~281--288.
  \MR{MR1222466}

\bibitem[Led01]{ledr01}
\bysame, \emph{Some asymptotic properties of random walks on free groups},
  Topics in probability and {L}ie groups: boundary theory, CRM Proc. Lecture
  Notes, vol.~28, Amer. Math. Soc., Providence, RI, 2001, pp.~117--152.
  \MR{MR1832436}

\bibitem[Led10]{ledr10}
\bysame, \emph{Linear drift and entropy for regular covers}, Geom. Funct. Anal.
  \textbf{20} (2010), 710--725. \MR{MR2720229}

\bibitem[L{\oe}u11]{Loe11}
Julien L{\oe}uillot, \emph{Entropie asymptotique, rayon spectral et vitesse de
  fuite des marches al\'eatoires dans les groupes}, M\'emoire de Master
  Recherche, Univ. Nantes, 2011.

\bibitem[MM07]{MaMa}
Jean Mairesse and Fr{\'e}d{\'e}ric Math{\'e}us, \emph{Random walks on free
  products of cyclic groups}, J. Lond. Math. Soc. (2) \textbf{75} (2007),
  47--66. \MR{MR2302729}

\bibitem[Nag97]{nagnibeda_rho}
Tatyana Nagnibeda, \emph{An upper bound for the spectral radius of a random
  walk on surface groups}, Zap. Nauchn. Sem. S.-Peterburg. Otdel. Mat. Inst.
  Steklov. (POMI) \textbf{240} (1997), 154--165, 293--294. \MR{MR1691645}

\bibitem[Var85]{Varop85}
Nicholas~Th. Varopoulos, \emph{Long range estimates for {M}arkov chains}, Bull.
  Sci. Math. (2) \textbf{109} (1985), 225--252. \MR{MR822826}

\bibitem[Woe89]{Woe89}
Wolfgang Woess, \emph{Boundaries of random walks on graphs and groups with
  infinitely many ends}, Israel J. Math. \textbf{68} (1989), 271--301.
  \MR{MR1039474}

\bibitem[Woe93]{Woe93}
\bysame, \emph{Fixed sets and free subgroups of groups acting on metric
  spaces}, Math. Z. \textbf{214} (1993), 425--439. \MR{MR1245204}

\bibitem[Woe00]{Woe00}
\bysame, \emph{Random walks on infinite graphs and groups}, Cambridge Tracts in
  Mathematics, vol. 138, Cambridge University Press, Cambridge, 2000.
  \MR{MR1743100}

\bibitem[Zim78]{Zim78}
Robert~J. Zimmer, \emph{Amenable ergodic group actions and an application to
  {P}oisson boundaries of random walks}, J. Functional Analysis \textbf{27}
  (1978), 350--372. \MR{MR0473096}

\end{thebibliography}
\bibliographystyle{amsalpha}
\end{document}